\documentclass[a4paper,12pt,reqno,english]{amsart}
\usepackage[utf8]{inputenc}
\usepackage[T1]{fontenc}
\usepackage{babel}

\usepackage[bookmarks = true, colorlinks, citecolor = blue, urlcolor = blue]{hyperref}
\usepackage[capitalize, nameinlink]{cleveref}
\usepackage[titletoc, toc, title]{appendix}

\usepackage{amssymb}
\usepackage{tikz}
\usepackage{tikz-cd}

\numberwithin{equation}{section}
\numberwithin{figure}{section}

\theoremstyle{plain}
\newtheorem{theorem}{Theorem}[section]

\theoremstyle{plain}
\newtheorem*{theorem*}{Theorem}

\theoremstyle{plain}
\newtheorem{proposition}[theorem]{Proposition}

\theoremstyle{plain}
\newtheorem{lemma}[theorem]{Lemma}

\theoremstyle{plain}
\newtheorem{corollary}[theorem]{Corollary}

\theoremstyle{definition}
\newtheorem{definition}[theorem]{Definition}

\theoremstyle{definition}

\theoremstyle{definition}

\theoremstyle{remark}

\addto\extrasenglish{
  
}

\newcommand{\lieg}{\mathfrak{g}}
\newcommand{\lieh}{\mathfrak{h}}
\newcommand{\liel}{\mathfrak{l}}
\newcommand{\liep}{\mathfrak{p}}

\newcommand{\Uqg}{U_q(\mathfrak{g})}
\newcommand{\UqlS}{U_q(\mathfrak{l}_S)}

\newcommand{\UqgZ}{U_q^\mathbb{Z}(\mathfrak{g})}

\newcommand{\CqG}{\mathbb{C}_q[G]}
\newcommand{\SGP}{S_q[G / P_S]}
\newcommand{\SGPop}{S_q[G / P_S^\mathrm{op}]}

\newcommand{\flagman}{G / P_S}
\newcommand{\Cflag}{\mathbb{C}_q[G / P_S]}

\newcommand{\id}{\mathrm{id}}

\newcommand{\diff}{\mathrm{d}}
\newcommand{\del}{\partial}
\newcommand{\delbar}{\overline{\partial}}

\newcommand{\algP}{\mathcal{A}_+}
\newcommand{\algM}{\mathcal{A}_-}
\newcommand{\algC}{\mathcal{A}_\mathbb{C}}
\newcommand{\alg}{\mathcal{A}}
\newcommand{\subalg}{\mathcal{B}}

\newcommand{\calcP}{\Gamma_{+}}
\newcommand{\calcM}{\Gamma_{-}}
\newcommand{\calcPC}{\Gamma_{+, \mathbb{C}}}
\newcommand{\calcMC}{\Gamma_{-, \mathbb{C}}}
\newcommand{\calcDel}{\Gamma_\del}
\newcommand{\calcDelb}{\Gamma_{\delbar}}
\newcommand{\calc}{\Gamma_{\diff}}
\newcommand{\calcUn}{\Gamma^\wedge_{\diff, \mathrm{u}}}

\newcommand{\wt}{\mathrm{wt}}

\newcommand{\takalg}{\Omega^\bullet / \subalg^+ \Omega^\bullet}

\newcommand{\algquo}{\Phi(\Omega^\bullet)}

\newcommand{\IndP}{I_{(1)}}

\newcommand{\calA}{\mathcal{A}}
\newcommand{\calB}{\mathcal{B}}
\newcommand{\calC}{\mathcal{C}}
\newcommand{\calH}{\mathcal{H}}
\newcommand{\calM}{\mathcal{M}}

\newcommand{\iu}{\mathrm{i}}

\newcommand{\braidz}{T}

\newcommand{\genf}{\mathsf{f}}
\newcommand{\genv}{\mathsf{v}}
\newcommand{\genc}{\mathsf{c}}

\newcommand{\laupol}{\mathbb{Z}[q^{1/m}, q^{-1/m}]}

\begin{document}

\title[Kähler structures on quantum irreducible flag manifolds]{Kähler structures on quantum irreducible flag manifolds}

\author{Marco Matassa}

\address{OsloMet – storbyuniversitetet}

\email{marco.matassa@oslomet.no}

\begin{abstract}
We prove that all quantum irreducible flag manifolds admit Kähler structures, as defined by Ó Buachalla.
In order to show this result, we also prove that the differential calculi defined by Heckenberger and Kolb are differential $*$-calculi in a natural way.
\end{abstract}

\maketitle

\section*{Introduction}

Within the realm of non-commutative geometry, the study of structures coming from complex geometry is a relatively new trend, see for instance the papers \cite{fgr99, besm13, obu16}.
Here we are interested in Kähler structures, which were defined recently in \cite{obu17}.
Recall that the existence of a Kähler structure on a complex manifold has many far-reaching consequences, see \cite{huy} for an overview.
As shown in \cite{obu17}, many of these consequences also hold in the non-commutative setting, provided they are reformulated accordingly.

The main problem then becomes to prove the existence of such Kähler structures.
In the paper \cite{obu17} it was shown that they do exist for the class of quantum projective spaces.
More generally, it was conjectured that they should exist for all quantum irreducible flag manifolds.
The aim of this paper is to answer this conjecture in the affirmative.

Recall that a (generalized) flag manifold is a homogeneous space of the form $G / P$, where $P$ is a parabolic subgroup of $G$.
These spaces admit natural Kähler structures and moreover they exhaust all compact homogeneous Kähler manifolds \cite{wan54}.
The condition of being irreducible is equivalent to $G / P$ being a symmetric space. Hence the class of irreducible flag manifolds coincides with that of irreducible compact Hermitian symmetric spaces.

Quantum flag manifolds can be defined straightforwardly in terms of quantum subgroups of the quantum groups $\CqG$, see \cite{stdi99}.
The class of quantum irreducible flag manifolds is singled out by a series of important results of Heckenberger and Kolb \cite{heko04, heko06}.
They show that these quantum spaces admit a canonical $q$-analogue of the de Rham complex, with the homogenous components having the same dimensions as in the classical case. We stress that this is definitely not the case for general quantum spaces.

Since the definition of a Kähler structure on a quantum space requires the existence of a differential calculus, quantum irreducible flag manifolds clearly provide the best avenue for testing this concept.
However there is an obstacle that needs to be overcome: to study the existence of Kähler structures we actually need to have a differential $*$-calculus, a structure which has not been introduced yet for the Heckenberger-Kolb calculi.

For this reason the paper contains two main results.
The first result is \cref{thm:star-calculi}, which shows that the Heckenberger-Kolb calculus $\Omega^\bullet$ over $\Cflag$ becomes a differential $*$-calculus in a natural way.
The second result is \cref{thm:kahler}, which shows the existence of a Kähler structure on $\Omega^\bullet$, thus proving the conjecture formulated in \cite[Conjecture 4.25]{obu17}.

These results provide some further steps in the general understanding of complex geometry within the quantum setting.
Of course, many more questions still remain to be answered.
As an example, the question of positive-definiteness of the quantum metric coming from the Kähler structure, as defined in \cite{obu17}, certainly deserves further study.

The organization of the paper is as follows.
In \cref{sec:quantum} we discuss various preliminaries related to quantized enveloping algebras and quantum coordinate rings.
In \cref{sec:calculi} we recall various basic definitions regarding differential calculi, as well as the notions of Hermitian and Kähler structures.
In \cref{sec:flag} we review the description of quantum flag manifolds in terms of generators and relations.
In \cref{sec:heko} we present the Heckenberger-Kolb calculi and we prove our first main result, namely that they are naturally differential $*$-calculi.
In \cref{sec:structures} we prove our second main result, namely the existence of Kähler structures for these differential $*$-calculi.
Finally in \cref{sec:braiding} and \cref{sec:braidingcalc} we prove various identities, mainly related to the braiding, which are used in the proofs of the main text.

\vspace{3mm}

{\footnotesize
\emph{Acknowledgements}.
I would like to thank Réamonn Ó Buachalla for his comments on this paper.
}

\section{Notations and preliminaries}
\label{sec:quantum}

In this section we recall some basic facts concerning quantized enveloping algebras and quantum coordinate rings, as well as fixing some notations.
More details and missing explanations can be found in textbooks such as \cite{klsc97, netu13}.

\subsection{Quantized enveloping algebras}

Let $\lieg$ be a complex simple Lie algebra, with Cartan subalgebra $\lieh$, and denote by $(\cdot, \cdot)$ the non-degenerate symmetric bilinear form on $\lieh^*$ induced by the Killing form. We denote by $\Uqg$ the \emph{quantized enveloping algebra} of $\lieg$, the Hopf algebra with generators $\{ K_i, E_i, F_i \}_{i = 1}^r$ and relations as in \cite{heko06}. In particular we have
\[
\begin{gathered}
\Delta(K_i) = K_i \otimes K_i, \quad
\Delta(E_i) = E_i \otimes K_i + 1 \otimes E_i, \quad
\Delta(F_i) = F_i \otimes 1 + K_i^{-1} \otimes F_i, \\
S(K_i) = K_i^{-1}, \quad
S(E_i) = - E_i K_i^{-1}, \quad
S(F_i) = - K_i F_i.
\end{gathered}
\]
With these conventions we have the identity $S^2(X) = K_{2 \rho} X K_{2 \rho}^{-1}$, where $\rho$ is the half-sum of the positive roots of $\lieg$.
We will consider $0 < q < 1$, so that we have a $*$-structure corresponding to the compact real form of $\lieg$. For instance we can take
\[
K_i^* = K_i, \quad
E_i^* = F_i K_i, \quad
F_i^* = K_i^{-1} E_i.
\]
We remark that the precise form of the $*$-structure will not matter in the following, hence we are free to replace it with any other equivalent one.

For $0 < q < 1$ the representation theory of $\Uqg$ is essentially the same as for $U(\lieg)$.
Hence for any dominant weight $\lambda$ we have a $\Uqg$-module $V(\lambda)$.
Recall that the dual space $V^*$ becomes a $\Uqg$-module by $(X f)(v) := f(S(X) v)$, where $v \in V$ and $f \in V^*$.

Next we will consider the braiding on the category of $\Uqg$-modules, namely a collection of $\Uqg$-module isomorphisms $\hat{R}_{V, W} : V \otimes W \to W \otimes V$, where $V$ and $W$ are $\Uqg$-modules.
We will follow the choice of \cite{heko06}: the braiding is uniquely determined by the requirement that $\hat{R}_{V, W}$ is a $\Uqg$-module isomorphism and by the condition
\[
\hat{R}_{V, W} (v_\mathrm{hw} \otimes w_\mathrm{lw}) := q^{(\wt v_\mathrm{hw}, \wt w_\mathrm{lw})} w_\mathrm{lw} \otimes v_\mathrm{hw},
\]
where $v_\mathrm{hw}$ is a highest weight vector of $V$ and $w_\mathrm{lw}$ is a lowest weight vector of $W$.
Choosing a basis $\{v_i\}_i$ of $V$ and a basis $\{w_j\}_j$ of $W$, we will write
\[
\hat{R}_{V, W} (v_i \otimes w_j) = \sum_{k, l} (\hat{R}_{V, W})^{k l}_{i j} w_k \otimes v_l.
\]

\subsection{Quantum coordinate rings}

Next we recall the \emph{quantum coordinate rings} $\CqG$, which are essentially the Hopf $*$-algebras duals to $\Uqg$.
They are obtained from the matrix coefficients of the finite-dimensional (type 1) representations of $\Uqg$.
Recall that, given a $\Uqg$-module $V$, the \emph{matrix coefficients} are given by
\[
c^V_{f, v} (X) := f(X v), \quad
f \in V^*, \ v \in V.
\]

By a \emph{$\Uqg$-invariant inner product} on $V$ we mean an inner product $(\cdot, \cdot) : V \times V \to \mathbb{C}$, conjugate-linear in the first variable, such that
\[
(X v, w) = (v, X^* w), \quad
\forall v, w \in V, \ \forall X \in \Uqg.
\]
It is well-known that if $V$ is a simple module then this inner product is unique, up to a constant.
Now fix $(\cdot, \cdot)$ on $V$ and take an orthonormal basis $\{v_i\}_i$. Then the elements of the dual basis $\{f_i\}_i$ of $V^*$ can be identified with $f_i = (v_i, \cdot)$. In this case we write
\[
u^V_{i j} := c^V_{f_i, v_j}.
\]
If $V = V(\lambda)$ is a simple $\Uqg$-module of highest weight $\lambda$, we will also write
\[
c^\lambda_{f, v} := c^{V(\lambda)}_{f, v}, \quad
u^\lambda_{i j} := u^{V(\lambda)}_{i j}.
\]
It is easy to see that the elements $u^V_{i j}$ satisfy the relations
\[
\Delta(u^V_{i j}) = \sum_k u^V_{i k} \otimes u^V_{k j}, \quad
(u^V_{i j})^* = S(u^V_{j i}).
\]
Moreover, since $S^2(X) = K_{2 \rho} X K_{2 \rho}^{-1}$ we have
\[
S^2(u^V_{i j}) = q^{(2 \rho, \wt v_i - \wt v_j)} u^V_{i j}.
\]

\section{Differential calculi}
\label{sec:calculi}

In this section we recall various notions related to the description of differential calculi on quantum spaces.
In particular we consider the notions of Hermitian and Kähler structures, as defined in \cite{obu17}.
Moreover we recall some aspects of Takeuchi's categorical equivalence, which is a quite useful tool when dealing with differential calculi.

\subsection{First order differential calculus}

A \emph{first order differential calculus} (FODC) over an algebra $\calA$ is an $\calA$-bimodule $\Gamma$ together with a linear map $\diff: \calA \to \Gamma$, such that $\Gamma = \mathrm{span} \{ a \diff(b) c : a, b, c \in \calA \}$ and $\diff$ satisfies the Leibnitz rule
\[
\diff(a b) = \diff(a) b + a \diff(b).
\]
If $\calA$ is a $*$-algebra, then $\Gamma$ is a \emph{$*$-FODC} if the $*$-structure of $A$ extends to a $*$-structure of $\Gamma$ in such a way that $\diff(a)^* = \diff(a^*)$.

Suppose in addition that $\calH$ is a Hopf algebra and $\Delta_\calA: \calA \to \calH \otimes \calA$ is a left $\calH$-comodule algebra structure on $\calA$.
Then $\Gamma$ is called \emph{left-covariant} if there exists a left $\calH$-comodule structure $\Delta_\Gamma: \Gamma \to \calH \otimes \Gamma$ on $\Gamma$ such that
\[
\Delta_\Gamma(a \diff(b) c) = \Delta_\calA(a) (\id \otimes \diff) (\Delta_\calA(b)) \Delta_\calA(c).
\]

Given a family of FODCs $\{ (\Gamma_i, \diff_i) \}_{i = 1}^n$, their \emph{direct sum} is the FODC $(\Gamma, \diff)$ with $\diff = \bigoplus_i d_i$ and $\Gamma = \calA \cdot \diff(\calA) \subset \bigoplus_i \Gamma_i$.
If the calculi are left covariant then so is their direct sum.

Finally, suppose that $\calB \subset \calA$ is a subalgebra and $(\Gamma, \diff)$ is a FODC over $\calA$. Then there is a FODC $(\Gamma |_\calB, \diff |_\calB)$ over $\calB$ defined by
\[
\Gamma |_\calB := \{ a \diff(b) : a, b \in \calB \}, \quad
\diff |_\calB (a) := \diff(a), \ \forall a \in \calB.
\]
This FODC is called the \emph{FODC over $\calB$ induced by $\Gamma$}.

\subsection{Higher order differential calculus}

A \emph{differential calculus} over $\calA$ is a differential graded algebra $(\Gamma^\wedge = \bigoplus_{k \in \mathbb{N}} \Gamma^{\wedge k}, \diff)$ such that $\Gamma^{\wedge 0} = \calA$ and $\Gamma^\wedge$ is generated by $\calA$ and $\diff \calA$.
We say that $\Gamma^\wedge$ has \emph{dimension} $n$ if $\Gamma^{\wedge n} \neq 0$ and $\Gamma^{\wedge k} = 0$ for $k > n$.

If $\calA$ is a $*$-algebra, then $\Gamma^\wedge$ is a \emph{differential $*$-calculus} if the $*$-structure of $\calA$ extends to an involution of $\Gamma$ such that $\diff(\omega)^* = \diff(\omega^*)$ for any $\omega \in \Gamma$ and moreover
\[
(\omega \wedge \chi)^* = (-1)^{p q} \chi^* \wedge \omega^*, \quad
\omega \in \Gamma^{\wedge p}, \ \chi \in \Gamma^{\wedge q}.
\]
An element $\omega \in \Gamma^\wedge$ is called \emph{real} if $\omega^* = \omega$.

Given a FODC $\Gamma$ over $\calA$, there exists a \emph{universal differential calculus} $(\Gamma^\wedge_{\mathrm{u}}, \diff_{\mathrm{u}})$, uniquely determined by the following property: if $\Gamma^\wedge$ is any differential calculus over $\calA$ such that $\Gamma^{\wedge 1} = \Gamma$, then $\Gamma$ is isomorphic to a quotient of $\Gamma^\wedge_{\mathrm{u}}$.
The $*$-structure lifts to the universal differential calculus, meaning that if $\Gamma$ is a $*$-FODC then $\Gamma^\wedge_{\mathrm{u}}$ is a differential $*$-calculus in a canonical way, see \cite[Chapter 12, Proposition 4]{klsc97}.

\subsection{Hermitian and Kähler structures}

Many structures from complex geometry can be adapted to the quantum setting, as discussed in \cite{obu17}.
We will now recall the notions of Hermitian and Kähler structures, as defined in the cited paper.
In this subsection $(\Omega^\bullet, \diff)$ will denote a differential $*$-calculus of dimension $2 n$.

\begin{definition}
\label{def:almost}
An \emph{almost symplectic form} is a central real 2-form $\sigma \in \Omega^\bullet$ satisfying the following property: denoting by $L_\sigma: \Omega^\bullet \to \Omega^\bullet$ the \emph{Lefschetz map} given by $L_\sigma(\omega) := \sigma \wedge \omega$, the map $L_\sigma^{n - k}: \Omega^k \to \Omega^{2 n - k}$ is an isomorphism for all $k = 0, \cdots, n - 1$.
\end{definition}

We will omit the subscript $\sigma$ in the following, as the dependence will be clear.

\begin{definition}
\label{def:hermitian}
A \emph{Hermitian structure} for $\Omega^\bullet$ is a pair $(\Omega^{(\bullet, \bullet)}, \sigma)$, where $\Omega^{(\bullet, \bullet)}$ is a complex structure and $\sigma$ is an almost symplectic form, called the \emph{Hermitian form}, such that $\sigma \in \Omega^{(1, 1)}$.
\end{definition}

For the definition of complex structures in this context we refer to \cite{obu17} and \cite{klvs11}.

\begin{definition}
\label{def:kahler}
A \emph{Kähler structure} for $\Omega^\bullet$ is a Hermitian structure $(\Omega^{(\bullet, \bullet)}, \kappa)$, such that the Hermitian form $\kappa$ is $\diff$-closed.
\end{definition}

The existence of such structures on a differential calculus $\Omega^\bullet$ has various consequences, as in the classical case.
We refer to \cite{obu17} for these results and more background material.

\subsection{Takeuchi's categorical equivalence}

We will now briefly review some aspects of Takeuchi's categorical equivalence \cite{takeuchi}, following \cite[Section 2.2.8]{heko06}.

Let $\calB \subset \calA$ be a left coideal subalgebra of a Hopf algebra $\calA$ with bijective antipode. Then $\calC := \calA / \calB^+ \calA$ is a right $\calA$-module coalgebra, where $\calB^+ := \{ b \in \calB : \varepsilon(b) = 0 \}$.
Let ${}^\calA_\calB \calM$ denote the category of left $\calA$-covariant left $\calB$-modules and let ${}^\calC \calM$ denote the category of left $\calC$-comodules.
Then there exist functors
\[
\begin{gathered}
\Phi: {}^\calA_\calB \calM \to {}^\calC \calM, \quad \Phi(\Gamma) := \Gamma / \calB^+ \Gamma, \\
\Psi: {}^\calC \calM \to {}^\calA_\calB \calM, \quad \Psi(V) := \calA \square_\calC V.
\end{gathered}
\]
Here $\square_\calC$ denotes the cotensor product over $\calC$.

\begin{theorem}[{\cite[Theorem 1]{takeuchi}}]
Suppose that $\calA$ is a faithfully flat right $\calB$-module. Then $\Phi$ and $\Psi$ give rise to an equivalence of categories between ${}^\calA_\calB \calM$ and ${}^\calC \calM$.
\end{theorem}

For all the algebras considered in this paper the condition of being faithfully flat will be satisfied, hence we will be able to use Takeuchi's categorical equivalence.

\section{Quantum flag manifolds}
\label{sec:flag}

In this section we define the quantum flag manifolds $\Cflag$ and recall their presentation by generators and relations, as given by Heckenberger and Kolb.
Moreover we will discuss the $*$-structure on $\Cflag$ and its action on the generators.

\subsection{Generators}
Let $S \subseteq \Pi$ be a subset of the simple roots of $\lieg$. Corresponding to any such choice we have the Levi factor $\liel_S$, which is a subalgebra of the standard parabolic subalgebra $\liep_S$.
In the quantum setting we define the \emph{quantized Levi factor} by
\[
\UqlS := \langle K_\lambda, \ E_i, \ F_i : i \in S \rangle \subseteq \Uqg.
\]
Here $\langle \cdot \rangle$ denotes the algebra generated by these elements in $\Uqg$.
Note that this is a Hopf $*$-subalgebra.
Then we define the \emph{quantum flag manifold} corresponding to $G / P_S$ by
\[
\Cflag := \CqG^{\UqlS} = \{ a \in \CqG : X a = \varepsilon(X) a, \ \forall X \in \UqlS \}.
\]

In the classical case the following realization of $\flagman$ is well-known.
Consider the dominant weight $\lambda := \sum_{i \notin S} \omega_i$ and write $N := \dim V(\lambda)$.
Then $G / P_S$ is isomorphic to the $G$-orbit of the highest weight vector $v_\lambda \in V(\lambda)$ in the projective space $\mathbb{P} (V(\lambda))$.

We will fix a weight basis $\{v_i\}_{i = 1}^N$ of $V(\lambda)$, with the convention that $v_N$ is a highest weight vector.
Denote by $\{f_i\}_{i = 1}^N$ the dual basis of $V(\lambda)^* \cong V(- w_0 \lambda)$. Then we define
\[
z^{i j} := c^\lambda_{f_i, v_N} c^{- w_0 \lambda}_{v_j, f_N} \in \CqG, \quad i, j = 1, \cdots, N.
\]
Here in writing $c^{- w_0 \lambda}_{v_j, f_N}$, we consider the element $v_j \in V(\lambda)$ as an element of $V(\lambda)^{**}$.

Let us explain this point in more detail, as it can give rise to some confusion.
Given a finite-dimensional $\Uqg$-module $V$, we have a map $V \to V^{**}$ which assign to $v \in V$ the linear functional $\widetilde{v} \in V^{**}$ given by $\widetilde{v}(f) := f(v)$, where $f \in V^*$.
The map $V \to V^{**}$ is a vector space isomorphism, but not an isomorphism of $\Uqg$-modules, since we have the relation
\[
X \widetilde{v}(f) = f(S^2(X) v) = \widetilde{S^2(X) v}(f).
\]
The upshot is that we can identify $V^{**}$ with $V$ in terms of the action $X \cdot v = S^2(X) v$.

\begin{proposition}[{\cite[Proposition 3.2]{heko06}}]
The elements $\{z^{i j}\}_{i, j = 1}^N$ generate $\Cflag$.
\end{proposition}

Observe that $\Cflag$ has a natural factorization in terms of the algebras
\[
\SGP := \{ c^\lambda_{f, v_N} : f \in V(\lambda)^* \}, \quad
\SGPop := \{ c^{- w_0 \lambda}_{v, f_N} : v \in V(\lambda)^{**} \}.
\]
These are the quantum analogues of the homogeneous coordinate rings of $G / P_S$ and $G / P_S^\mathrm{op}$.

\subsection{Relations}

In order to present the relations for the quantum flag manifolds $\Cflag$, it is convenient to introduce some auxiliary algebras.
First we define two algebras $\algP$ and $\algM$ as follows.
The algebra $\algP$ has generators $\genf^1, \cdots, \genf^N$ and relations
\[
\genf^i \genf^j = q^{-(\lambda, \lambda)} \sum_{k, l} (\hat{R}_{V, V})^{i j}_{k l} \genf^k \genf^l.
\]
The algebra $\algM$ has generators $\genv^1, \cdots, \genv^N$ and relations
\[
\genv^i \genv^j = q^{-(\lambda, \lambda)} \sum_{k, l} (\hat{R}_{V^*, V^*})^{i j}_{k l} \genv^k \genv^l.
\]
They become $\Uqg$-module algebras by identifying $\{\genf^i\}_{i = 1}^N$ with the basis $\{f_i\}_{i = 1}^N$ of $V(\lambda)^*$ and $\{\genv^i\}_{i = 1}^N$ with the basis $\{v_i\}_{i = 1}^N$ of $V(\lambda)^{**}$.
In this way we obtain $\Uqg$-module algebra isomorphisms between $\SGP$ and $\algP$ and between $\SGPop$ and $\algM$, given by
\[
\genf^i \mapsto c^\lambda_{f_i, v_N}, \quad
\genv^i \mapsto c^{- w_0 \lambda}_{v_i, f_N}.
\]

Next we define the algebra $\algC := \algP \otimes \algM$ with the exchange relations
\[
\genv^i \genf^j := q^{(\lambda, \lambda)} \sum_{k, l} (\hat{R}_{V, V^*})^{i j}_{k l} \genf^k \genv^l.
\]
The algebra $\algC$ admits a central invariant element defined by
\[
\genc := \sum_i \genv^i \genf^i.
\]
Hence we can define $\alg := \alg_\mathbb{C} / \langle \genc - 1 \rangle$.
It becomes a $\mathbb{Z}$-graded algebra upon setting $\deg \genf^i = -1$ and $\deg \genv^i = 1$.
Finally we write $\subalg := \alg^0$ for the degree zero part of $\alg$.

\begin{proposition}[{\cite[Proposition 3.2]{heko06}}]
\label{prop:iso-alg}
We have an isomorphism of $\Uqg$-module algebras $\subalg \cong \Cflag$ given by $\genf^i \genv^j \mapsto z^{i j}$.
\end{proposition}

\subsection{*-structure}

The quantum flag manifolds $\Cflag$ are naturally $*$-algebras, since they are defined by invariance with respect to the Hopf $*$-algebras $\UqlS$.
This $*$-structure can be transported to the algebras $\alg$ and $\subalg$.
However we will introduce a $*$-structure on $\alg$ from scratch, as a warm-up to the case of the $*$-calculi to be discussed in the next section.

To make our life easier, we will assume from now on that the basis $\{v_i\}_{i = 1}^N$ of $V(\lambda)$ is orthonormal with respect to a $\Uqg$-invariant inner product.

\begin{proposition}
The algebras $\alg$ and $\subalg$ become $*$-algebras upon setting $\genf^{i *} = \genv^i$.
\end{proposition}

\begin{proof}
First we check that we obtain a $*$-structure on $\algC$ in this way.
It suffices to check that the relations of $\algC$ are preserved under $*$. We compute
\[
\begin{split}
(\genf^j \genf^i)^* & = \genv^i \genv^j = q^{-(\lambda, \lambda)} \sum_{k, l} (\hat{R}_{V^*, V^*})^{i j}_{k l} \genv^k \genv^l \\
& = q^{-(\lambda, \lambda)} \sum_{k, l} \overline{(\hat{R}_{V, V})^{j i}_{l k}} \genv^k \genv^l = \left( q^{-(\lambda, \lambda)} \sum_{k, l} (\hat{R}_{V, V})^{j i}_{l k} \genf^l \genf^k \right)^*,
\end{split}
\]
where we have used the first identity of \cref{prop:braiding-components}, which we note requires the use of the orthonormal basis $\{v_i\}_{i = 1}^N$.
Similarly for the relation between the $\genv^i$ generators.

Next we look at the cross-relations. We compute
\[
\begin{split}
(\genv^j \genf^i)^* & = \genv^i \genf^j = q^{(\lambda, \lambda)} \sum_{k, l} (\hat{R}_{V, V^*})^{i j}_{k l} \genf^k \genv^l \\
& = q^{(\lambda, \lambda)} \sum_{k, l} \overline{(\hat{R}_{V, V^*})^{j i}_{l k}} \genf^k \genv^l = \left( q^{(\lambda, \lambda)} \sum_{k, l} (\hat{R}_{V, V^*})^{j i}_{l k} \genf^l \genv^k \right)^*,
\end{split}
\]
where we have used the second identity of \cref{prop:braiding-components}.
Finally this $*$-structure descends to the quotient $\alg = \alg_\mathbb{C} / \langle \genc - 1 \rangle$, since we have $\genc^* = \genc$.
\end{proof}

Now we show that this $*$-structure agrees with the one on $\Cflag$.

\begin{corollary}
\label{cor:star-iso}
The map $\genf^i \genv^j \mapsto z^{i j}$ is an isomorphism of $\Uqg$-module $*$-algebras.
\end{corollary}

\begin{proof}
We have $(\genf^i \genv^j)^* = \genf^j \genv^i$, hence it suffices to show that $(z^{i j})^* = z^{j i}$.
First we claim that $S(c^V_{f, v}) = c^{V^*}_{\widetilde{v}, f}$, where we use the notation $\widetilde{v}(f) = f(v)$. Indeed we compute
\[
S(c^V_{f, v})(X) = c^V_{f, v}(S(X)) = f(S(X) v) = (X f)(v) = \widetilde{v}(X f) = c^{V^*}_{\widetilde{v}, f}(X).
\]
Hence $z^{i j} = c^\lambda_{f_i, v_N} S(c^\lambda_{f_N, v_j})$. Moreover, since $\{v_i\}_{i = 1}^N$ is an orthonormal basis, we can write $u^\lambda_{i j} = c^\lambda_{f_i, v_j}$ and use $(u^V_{i j})^* = S(u^V_{j i})$. Therefore $z^{i j} = u^\lambda_{i N} (u^\lambda_{j N})^*$, which shows $(z^{i j})^* = z^{j i}$.
\end{proof}

\section{Heckenberger-Kolb calculi and $*$-structures}
\label{sec:heko}

In this section we will consider the Heckenberger-Kolb calculus over $\subalg \cong \Cflag$, where $\flagman$ is an irreducible flag manifold.
We will show that this calculus is naturally a $*$-calculus, where the $*$-structure on $\subalg \subset \alg$ is the one introduced in the previous section.
After giving a brief presentation of the Heckenberger-Kolb calculi, we give the necessary definitions for the two FODCs $\calcDel$ and $\calcDelb$, whose direct sum gives the FODC $\calc$.
To show that $\calc$ is a $*$-calculus we will check that the relations are compatible with the $*$-structure of $\subalg$.

\subsection{The calculi in brief}

From this point on we will restrict to the case of \emph{irreducible} flag manifolds $\flagman$.
These spaces can be characterized by the following condition: we have $S = \Pi \backslash \{ \alpha_s \}$ and the simple root $\alpha_s$ has multiplicity $1$ in the highest root of $\lieg$.
In the following the index $s$ will always be associated to the simple root $\alpha_s$ removed from the set $S$.
Observe that in this case we have $\lambda = \omega_s$, where $\lambda$ is defined as in \cref{sec:flag}.

In the paper \cite{heko04}, Heckenberger and Kolb show that there exist exactly two non-isomorphic covariant FODCs over $\Cflag$. We denote them by $\calcDel$ and $\calcDelb$, as they classically correspond to the holomorphic and anti-holomorphic calculi on the complex manifold $\flagman$, and write $\calc = \calcDel \oplus \calcDelb$ for their direct sum.
In the follow-up paper \cite{heko06}, they investigate the universal differential calculi $\calcUn$ built from $\calc$.
They show that these calculi have classical dimensions and have many of the features of the classical calculi over $\flagman$.

Before giving the details, let us first outline the main steps of this construction.
First we define a FODC $\calcP$ over $\algP$. Then using $\calcP$ we construct a FODC $\calcPC$ over $\algC$.
By taking an appropriate quotient, we obtain a FODC $\calcPC / \Lambda_+$ over $\alg$.
Finally the calculus $\calcDel$ over $\subalg$ is simply the calculus induced by $\calcPC / \Lambda_+$ over $\alg$.
A similar construction gives $\calcDelb$ starting from $\algM$.
Hence we obtain the FODC $\calc$ over $\subalg$ as the direct sum of $\calcDel$ and $\calcDelb$.

Therefore we get the universal differential calculi $\calcUn$, which we will also denote by $\Omega^\bullet$.
Hence $\diff = \del + \delbar$ is a differential and we have the relation $\del \circ \delbar = - \delbar \circ \del$.

\subsection{The FODC $\calcDel$}

First we present the construction of the FODC $\calcDel$ over $\subalg$.
We start from the left $\algP$-module $\calcP$ generated by the elements $\{ \diff \genf^i \}_{i = 1}^N$ and the relations
\begin{equation}
\label{eq:GammaP-rels}
\sum_{i, j} (\hat{P}_V \hat{Q}_V)^{k l}_{i j} \genf^i \diff \genf^j = 0.
\end{equation}
Here we make use of the notations
\[
\hat{P}_V := \hat{R}_{V, V} - q^{(\omega_s, \omega_s)} \id, \quad
\hat{Q}_V := \hat{R}_{V, V} + q^{(\omega_s, \omega_s) - (\alpha_s, \alpha_s)} \id.
\]
We can make $\calcP$ into an $\algP$-bimodule by setting
\begin{equation}
\label{eq:GammaP-bimod}
(\diff \genf^i) \genf^j = q^{(\alpha_s, \alpha_s) - (\omega_s, \omega_s)} \sum_{k, l} (\hat{R}_{V, V})^{i j}_{k l} \genf^k \diff \genf^l.
\end{equation}

Next we consider the left $\algC$-module $\calcPC$ defined by
\[
\calcPC := \algC \otimes_{\algP} \calcP.
\]
It becomes an $\algC$-bimodule by setting
\begin{equation}
\label{eq:GammaPC-bimod}
(\diff \genf^i) \genv^j = q^{- (\omega_s, \omega_s)} \sum_{k, l} (\hat{R}^{-1}_{V, V^*})^{i j}_{k l} \genv^k \diff \genf^l.
\end{equation}
We have a differential $\del : \algC \to \calcPC$ given by
\[
\del(\genf^i) = \diff \genf^i, \quad \del(\genv^i) = 0.
\]
In the following we will always write $\diff \genf^i = \del(\genf^i)$.

Let $\Lambda_+ \subset \calcPC$ be the sub-bimodule generated by $\del(\genc)$, $(\genc - 1) \calcPC$ and $\calcPC (\genc - 1)$.
Then the quotient $\calcPC / \Lambda_+$ is a FODC over $\alg$. Finally $\calcDel$ is the FODC over $\subalg$ induced by $\calcPC / \Lambda_+$.

\subsection{The FODC $\calcDelb$}

Next we introduce the second covariant FODC $(\calcDelb, \delbar)$ over $\alg$.
Consider the left $\algM$-module $\calcM$ generated by the elements $\{ \diff \genv^i \}_{i = 1}^N$ and the relations
\begin{equation}
\label{eq:GammaM-rels}
\sum_{i, j} (\hat{P}_{V^*} \hat{Q}_{V^*})^{k l}_{i j} \genv^i \diff \genv^j = 0.
\end{equation}
Here we are using the notations
\[
\hat{P}_{V^*} := \hat{R}_{V^*, V^*} - q^{(\omega_s, \omega_s)} \id, \quad
\hat{Q}_{V^*} := \hat{R}_{V^*, V^*} + q^{(\omega_s, \omega_s) - (\alpha_s, \alpha_s)} \id.
\]
We turn the left $\algM$-module $\calcM$ into an $\algM$-bimodule by
\begin{equation}
\label{eq:GammaM-bimod}
(\diff \genv^i) \genv^j = q^{(\omega_s, \omega_s) - (\alpha_s, \alpha_s)} \sum_{k, l} (\hat{R}^{-1}_{V^*, V^*})^{i j}_{k l} \genv^k \diff \genv^l.
\end{equation}

Next we consider the left $\algC$-module $\calcMC$ defined by
\[
\calcMC := \algC \otimes_{\algM} \calcM.
\]
It becomes an $\algC$-bimodule by setting
\begin{equation}
\label{eq:GammaMC-bimod}
(\diff \genv^i) \genf^j = q^{(\omega_s, \omega_s)} \sum_{k, l} (\hat{R}_{V, V^*})^{i j}_{k l} \genf^k \diff \genv^l.
\end{equation}
We have a differential $\delbar : \algC \to \calcMC$ given by
\[
\delbar(\genf^i) = 0, \quad \delbar(\genv^i) = \diff \genv^i.
\]
In the following we will always write $\diff \genv^i = \delbar(\genv^i)$.

Let $\Lambda_- \subset \calcMC$ be the sub-bimodule generated by $\del(\genc)$, $(\genc - 1) \calcMC$ and $\calcMC (\genc - 1)$.
Then the quotient $\calcMC / \Lambda_-$ is a FODC over $\alg$. Finally $\calcDelb$ is the FODC over $\subalg$ induced by $\calcMC / \Lambda_-$.

\subsection{Differential *-calculus}

We will now investigate whether the FODC $\calc$ over the $*$-algebra $\subalg$ can be made into a differential $*$-calculus.
We start with a simple lemma.

\begin{lemma}
\label{lem:equivalent-Gamma}
The relations \eqref{eq:GammaP-rels}, \eqref{eq:GammaP-bimod} and \eqref{eq:GammaPC-bimod} in $\calcPC$ are equivalent to
\[
\begin{gathered}
\sum_{i, j} (\hat{P}_V \hat{Q}_V )^{k l}_{i j} \del(\genf^i) \genf^j = 0, \\
\genf^i \del(\genf^j) = q^{(\omega_s, \omega_s) - (\alpha_s, \alpha_s)} \sum_{k, l} (\hat{R}^{-1}_{V, V})^{i j}_{k l} \del(\genf^k) \genf^l, \\
\genv^i \del(\genf^j) = q^{(\omega_s, \omega_s)} \sum_{k, l} (\hat{R}_{V, V^*})^{i j}_{k l} \del(\genf^k) \genv^l.
\end{gathered}
\]
Similarly the relations \eqref{eq:GammaM-rels}, \eqref{eq:GammaM-bimod} and \eqref{eq:GammaMC-bimod} in $\calcMC$ are equivalent to
\[
\begin{gathered}
\sum_{i, j} (\hat{P}_{V^*} \hat{Q}_{V^*})^{k l}_{i j} \delbar(\genv^i) \genv^j = 0, \\
\genv^i \delbar(\genv^j) = q^{(\alpha_s, \alpha_s) - (\omega_s, \omega_s)} \sum_{k, l} (\hat{R}_{V^*, V^*})^{i j}_{k l} \delbar(\genv^k) \genv^l, \\
\genf^i \delbar(\genv^j) = q^{-(\omega_s, \omega_s)} \sum_{k, l} (\hat{R}^{-1}_{V, V^*})^{i j}_{k l} \delbar(\genv^k) \genf^l.
\end{gathered}
\]
\end{lemma}

\begin{proof}
Most of the identities follow straightforwardly by applying the inverse of the appropriate braiding.
The only non-trivial identities are those following from \eqref{eq:GammaP-rels} and \eqref{eq:GammaM-rels}.
Let us consider the first one. Plugging the identity for $\genf^i \del(\genf^j)$ into \eqref{eq:GammaP-rels} we obtain
\[
q^{(\omega_s, \omega_s) - (\alpha_s, \alpha_s)} \sum_{i, j} (\hat{P}_V \hat{Q}_V \hat{R}^{-1}_{V, V})^{k l}_{i j} \del(\genf^i) \genf^j = 0.
\]
Then using the relation $\hat{P}_V \hat{Q}_V \hat{R}^{-1}_{V, V} = \hat{R}^{-1}_{V, V} \hat{P}_V \hat{Q}_V$ and multiplying on the left by $\hat{R}_{V, V}$ we obtain the result.
The second identity is obtained similarly.
\end{proof}

We are now ready to prove the main result of this section.

\begin{theorem}
\label{thm:star-calculi}
The differential calculus $(\calcUn, \diff)$ over $\subalg$ is a differential $*$-calculus.
\end{theorem}

\begin{proof}
We have already mentioned that it suffices to show that the FODC $(\calc, \diff)$ is a $*$-FODC.
Consider the requirement $\diff(a)^* = \diff(a^*)$. Using $\genf^{i *} = \genv^i$, together with the relations $\diff(\genf^i) = \del(\genf^i)$ and $\diff(\genv^i) = \delbar(\genv^i)$, we immediately find $\del(\genf^i)^* = \delbar(\genv^i)$.
Now we have to check that the relations in $\Gamma$ are preserved by this candidate $*$-structure.

We start with the relations \eqref{eq:GammaP-rels} and \eqref{eq:GammaM-rels}.
Observe that $\overline{(\hat{P}_V)^{k l}_{i j}} = (\hat{P}_{V^*})^{l k}_{j i}$ and $\overline{(\hat{Q}_V)^{k l}_{i j}} = (\hat{Q}_{V^*})^{l k}_{j i}$, which follows from their definitions and \cref{prop:braiding-components}. Then we have
\[
\left( \sum_{i, j} (\hat{P}_V \hat{Q}_V)^{k l}_{i j} \genf^i \del(\genf^j) \right)^* = \sum_{i, j} (\hat{P}_{V^*} \hat{Q}_{V^*})^{l k}_{j i} \delbar(\genv^j) \genv^i = 0,
\]
where the last identity follows from \cref{lem:equivalent-Gamma}. Similarly for the relation \eqref{eq:GammaM-rels}.

Next consider the relations \eqref{eq:GammaP-bimod} and \eqref{eq:GammaM-bimod}. Using \cref{lem:equivalent-Gamma} we compute
\[
\begin{split}
(\del(\genf^i) \genf^j)^* & = \genv^j \delbar(\genv^i) = q^{(\alpha_s, \alpha_s) - (\omega_s, \omega_s)} \sum_{k, l} (\hat{R}_{V^*, V^*})^{j i}_{k l} \delbar(\genv^k) \genv^l \\
& = \left( q^{(\alpha_s, \alpha_s) - (\omega_s, \omega_s)} \sum_{k, l} (\hat{R}_{V, V})^{i j}_{l k} \genf^l \del(\genf^k) \right)^*,
\end{split}
\]
where we have used $\overline{(\hat{R}_{V, V})^{k l}_{i j}} = (\hat{R}_{V^*, V^*})^{l k}_{j i}$ from \cref{prop:braiding-components}.

Next let us consider \eqref{eq:GammaPC-bimod} and \eqref{eq:GammaMC-bimod}. Using \cref{lem:equivalent-Gamma} we compute
\[
\begin{split}
(\del(\genf^i) \genv^j)^* & = \genf^j \delbar(\genv^i) = q^{-(\omega_s, \omega_s)} \sum_{k, l} (\hat{R}^{-1}_{V, V^*})^{j i}_{k l} \delbar(\genv^k) \genf^l \\
& = \left( q^{-(\omega_s, \omega_s)} \sum_{k, l} (\hat{R}^{-1}_{V, V^*})^{i j}_{l k} \genv^l \del(\genf^k) \right)^*,
\end{split}
\]
where we have used $\overline{(\hat{R}^{-1}_{V, V^*})^{k l}_{i j}} = (\hat{R}^{-1}_{V, V^*})^{l k}_{j i}$ from \cref{prop:braiding-components}.

Finally we have to check that the sub-bimodules $\Lambda_+$ and $\Lambda_-$ are preserved under the $*$-structure. But this is clear, since we have $\genc^* = \genc$.
\end{proof}

\section{Hermitian and Kähler structures}
\label{sec:structures}

In this section we will show the existence of Hermitian and Kähler structures on the Heckenberger-Kolb calculi $\Omega^\bullet$ over quantum irreducible flag manifolds $\Cflag$.

\subsection{Some identities}

Recall the $\Uqg$-module algebra isomorphism $\subalg \cong \Cflag$ from \cref{prop:iso-alg}.
By a small abuse of notation, we will also write $z^{i j} = \genf^i \genv^j$ for the generators of the algebra $\subalg$.
We will now obtain some identities for these elements.

\begin{lemma}
We have $\sum_k z^{i k} z^{k j} = z^{i j}$ and $(z^{i j})^* = z^{j i}$.
\end{lemma}

\begin{proof}
The first follows from $\sum_i \genv^i \genf^i = 1$ while the second follows from $\genf^{i *} = \genv^i$.
\end{proof}

Next we consider some identities involving the differentials.

\begin{lemma}
\label{lem:del-delbar}
We have the identities
\[
\begin{gathered}
\sum_k z^{i k} \del(z^{k j}) = 0, \quad
\sum_k \del(z^{i k}) z^{k j} = \del(z^{i j}), \\
\sum_k z^{i k} \delbar(z^{k j}) = \delbar(z^{i j}), \quad
\sum_k \delbar(z^{i k}) z^{k j} = 0.
\end{gathered}
\]
\end{lemma}

\begin{proof}
Consider the identities for the differential $\del$. We compute
\[
\begin{split}
\sum_k \del(z^{i k}) z^{k j} & = \sum_k \del(\genf^i \genv^k) \genf^k \genv^j = \sum_k \del(\genf^i) \genv^k \genf^k \genv^j \\
& = \del(\genf^i) \genv^j = \del(\genf^i \genv^j) = \del(z^{i j}).
\end{split}
\]
On the other hand, using the fact that $z$ is a projection, we have
\[
\del(z^{i j}) = \sum_k \del(z^{i k} z^{k j}) = \sum_k \del(z^{i k}) z^{k j} + \sum_k z^{i k} \del(z^{k j}).
\]
Using the previous identity, this implies the vanishing of last term.
The corresponding identities for the differential $\delbar$ are obtained in a similar manner.
\end{proof}

The last identities we will require are the vanishing of certain degree $3$ terms.
We will use from now on the notation $\lambda_i := \wt(v_i)$.

\begin{lemma}
\label{lem:triple-wedge}
We have the identities
\[
\begin{split}
\sum_{i, j, k} q^{(2 \rho, \lambda_i)} \del z^{i j} \wedge \delbar z^{j k} \wedge \del z^{k i} & = 0, \\
\sum_{i, j, k} q^{(2 \rho, \lambda_i)} \delbar z^{i j} \wedge \del z^{j k} \wedge \delbar z^{k i} & = 0.
\end{split}
\]
\end{lemma}

\begin{proof}
To prove the first identity let us write
\[
Z_1 = \sum_{i, j, k} q^{(2 \rho, \lambda_i)} \del z^{i j} \wedge \delbar z^{j k} \wedge \del z^{k i}
= \sum_{i, j, k, l} q^{(2 \rho, \lambda_i)} \del z^{i j} \wedge \delbar z^{j k} \wedge \del(z^{k l}) z^{l i},
\]
where in the second equality we have used \cref{lem:del-delbar}.
We will move the element $z^{l i}$ to the left using the relations given in \cref{sec:braidingcalc}.
We obtain
\[
Z_1 = q^{(\alpha_s, \alpha_s)} \sum_{i, j, k, l} \sum_{\{ a, b, c, d \}} q^{(2 \rho, \lambda_i)} T_{a_1 b_1 c_1 d_1}^{i j a_2 b_2} T_{a_2 b_2 c_2 d_2}^{j k a_3 b_3} T_{a_3 b_3 c_3 d_3}^{k l l i} z^{a_1 b_1} \del z^{c_1 d_1} \wedge \delbar z^{c_2 d_2} \wedge \del z^{c_3 d_3},
\]
where ${\{ a, b, c, d \}}$ denotes the sum over the variables ${\{ a_i, b_i, c_i, d_i \}}$ with $i = 1, 2, 3$.
The product of the $T_{a b c d}^{i j k l}$ terms can be simplified using \cref{lem:idT1}. We compute
\[
\sum_{j, a_2, b_2} \sum_{k, a_3, b_3} T_{a_1 b_1 c_1 d_1}^{i j a_2 b_2} T_{a_2 b_2 c_2 d_2}^{j k a_3 b_3} T_{a_3 b_3 c_3 d_3}^{k l l i}
= \delta_{d_2, c_3} \sum_{j, a_2, b_2} T_{a_1 b_1 c_1 d_1}^{i j a_2 b_2} T_{a_2 b_2 c_2 d_3}^{j l l i}
= \delta_{d_1, c_2} \delta_{d_2, c_3} T_{a_1 b_1 c_1 d_3}^{i l l i}.
\]
Plugging this in we find
\[
Z_1 = q^{(\alpha_s, \alpha_s)} \sum_{i, l} \sum_{a_1, b_1, c_1} \sum_{d_1, d_2, d_3} q^{(2 \rho, \lambda_i)} T_{a_1 b_1 c_1 d_3}^{i l l i} z^{a_1 b_1} \del z^{c_1 d_1} \wedge \delbar z^{d_1 d_2} \wedge \del z^{d_2 d_3}.
\]
Next using the relation \cref{lem:idT2} we obtain
\[
Z_1 = q^{(\alpha_s, \alpha_s)} \sum_{a_1, b_1} \sum_{d_1, d_2} q^{(2 \rho, \lambda_{a_1})} z^{a_1 b_1} \del z^{b_1 d_1} \wedge \delbar z^{d_1 d_2} \wedge \del z^{d_2 a_1}.
\]
But then using $\sum_k z^{i k} \del z^{k j} = 0$ we conclude that $Z_1 = 0$.

The second identity is proven similarly. Let us write
\[
Z_2 = \sum_{i, j, k, l} q^{(2 \rho, \lambda_i)} \delbar z^{i j} \wedge \del z^{j k} \wedge \delbar(z^{k l}) z^{l i}.
\]
Note that $Z_2 = 0$ due to \cref{lem:del-delbar}.
We rewrite it in the form
\[
Z_2 = q^{-(\alpha_s, \alpha_s)} \sum_{i, j, k, l} \sum_{\{ a, b, c, d \}} q^{(2 \rho, \lambda_i)} T_{a_1 b_1 c_1 d_1}^{i j a_2 b_2} T_{a_2 b_2 c_2 d_2}^{j k a_3 b_3} T_{a_3 b_3 c_3 d_3}^{k l l i} z^{a_1 b_1} \delbar z^{c_1 d_1} \wedge \del z^{c_2 d_2} \wedge \delbar z^{c_3 d_3}.
\]
By the same type of computations we did for $Z_1$ we obtain
\[
Z_2 = q^{-(\alpha_s, \alpha_s)} \sum_{a_1, b_1} \sum_{d_1, d_2} q^{(2 \rho, \lambda_{a_1})} z^{a_1 b_1} \delbar z^{b_1 d_1} \wedge \del z^{d_1 d_2} \wedge \delbar z^{d_2 a_1}.
\]
Using \cref{lem:del-delbar} this can be rewritten as
\[
Z_2 = q^{-(\alpha_s, \alpha_s)} \sum_{i, j, k} q^{(2 \rho, \lambda_i)} \delbar z^{i j} \wedge \del z^{j k} \wedge \delbar z^{k i}.
\]
But since $Z_2 = 0$ this implies the second identity.
\end{proof}

\subsection{The Kähler form}

We will now introduce a $2$-form $\kappa \in \Omega^\bullet$, which we will later show to satisfy the conditions defining a Kähler form. It is defined by
\[
\kappa := \iu \sum_{i, j, k} q^{(2 \rho, \lambda_i)} z^{i j} \diff z^{j k} \wedge \diff z^{k i}.
\]
Here $\iu := \sqrt{-1}$ denotes the imaginary unit and $\lambda_i = \wt(v_i)$.

We begin by showing that $\kappa$ is left $\CqG$-coinvariant. In this proof we will consider the $z^{i j}$ as elements of $\Cflag$, so that $\Delta$ denotes the coproduct of $\CqG$.

\begin{lemma}
\label{lem:kappa-coinvariant}
The element $\kappa$ is left $\CqG$-coinvariant.
\end{lemma}

\begin{proof}
We have $z^{i j} = u^\lambda_{i N} S(u^\lambda_{N j})$, as seen in the proof of \cref{cor:star-iso}. We will omit the superscript $\lambda$ for notational convenience. Then we easily compute
\[
\Delta(z^{i j}) = \sum_{a, b} u_{i a} S(u_{b j}) \otimes z^{a b},
\]
in accordance with the fact that $\Cflag$ is a left coideal.
Next recall that in a left-covariant FODC $\Gamma$ we have $\Delta_\Gamma(a \diff b) = \Delta(a) (\id \otimes \diff) (\Delta(b))$.
Then we obtain
\[
\Delta_\Gamma (\kappa) = \iu \sum_{i, j, k} \sum_{a, b, c, d, e, f} q^{(2 \rho, \lambda_i)} u_{i a} S(u_{b j}) u_{j c} S(u_{d k}) u_{k e} S(u_{f i}) \otimes z^{a b} \diff z^{c d} \wedge \diff z^{e f}.
\]
Next we use the identities following from the Hopf algebra structure
\[
\sum_{k = 1}^N S(u_{i k}) u_{k j} = \delta_{i j}, \quad
\sum_{k = 1}^N q^{(2 \rho, \wt v_k - \wt v_j)} u_{k j} S(u_{i k}) = \delta_{i j},
\]
where the second one follows from the fact that $S^2(u_{i j}) = q^{(2 \rho, \wt v_i - \wt v_j)} u_{i j}$. Finally
\[
\begin{split}
\Delta_\Gamma (\kappa) & = \iu \sum_i \sum_{a, b, d, f} q^{(2 \rho, \lambda_i)} u_{i a} S(u_{f i}) \otimes z^{a b} \diff z^{b d} \wedge \diff z^{d f} \\
& = \iu \sum_{a, b, d} q^{(2 \rho, \lambda_a)} 1 \otimes z^{a b} \diff z^{b d} \wedge \diff z^{d a}.
\end{split}
\]
Therefore we get $\Delta_\Gamma (\kappa) = 1 \otimes \kappa$, that is $\kappa$ is left $\CqG$-coinvariant.
\end{proof}

We can now easily show that the $2$-form $\kappa$ satisfies most of the requirements of a Kähler form, as described in \cref{def:kahler}.

\begin{proposition}
\label{prop:kappa-props}
The $2$-form $\kappa$ satisfies the following properties:
\begin{enumerate}
\item it is closed,
\item it belongs to $\Omega^{(1, 1)}$,
\item it is central and real.
\end{enumerate}
\end{proposition}

\begin{proof}
(1) Using the identities in \cref{lem:del-delbar} it is easy to see that
\[
\begin{split}
\diff \kappa & = \iu \sum_{i, j, k} q^{(2 \rho, \lambda_i)} \diff z^{i j} \wedge \diff z^{j k} \wedge \diff z^{k i} \\
& = \iu \sum_{i, j, k} q^{(2 \rho, \lambda_i)} \del z^{i j} \wedge \delbar z^{j k} \wedge \del z^{k i}
+ \iu \sum_{i, j, k} q^{(2 \rho, \lambda_i)} \delbar z^{i j} \wedge \del z^{j k} \wedge \delbar z^{k i}.
\end{split}
\]
But it follows from \cref{lem:triple-wedge} that both terms vanish, hence $\diff \kappa = 0$.

(2) Again using \cref{lem:del-delbar} it is immediate to see that
\[
\kappa = \iu \sum_{i, j, k} q^{(2 \rho, \lambda_i)} z^{i j} \delbar z^{j k} \wedge \del z^{k i} = \iu \sum_{i, j} q^{(2 \rho, \lambda_i)} \delbar z^{i j} \wedge \del z^{j i}.
\]
This implies that $\kappa \in \Omega^{(1, 1)}$ for the natural complex structure on $\Omega^\bullet$.

(3) It is a general fact that every left $\CqG$-coinvariant $\diff$-closed form is central, see \cite[Corollary 4.6]{obu17} (this result requires $\calB^+ \Omega^1 = \Omega^1 \calB^+$, which holds in our case as explained in the next subsection).
Since $\kappa$ is left $\CqG$-coinvariant by \cref{lem:kappa-coinvariant} and $\diff$-closed, we conclude that it is central.
To show that $\kappa$ is real we use $(z^{i j})^* = z^{j i}$ and compute
\[
\kappa^* = \iu \sum_{i, j} q^{(2 \rho, \lambda_i)} \del(z^{j i})^* \wedge \delbar(z^{i j})^* = \iu \sum_{i, j} q^{(2 \rho, \lambda_i)} \delbar(z^{i j}) \wedge \del(z^{j i}) = \kappa. \qedhere
\]
\end{proof}

To finish the proof that $\kappa$ is a Kähler form we need to check the condition on the Lefschetz map $L: \Omega^\bullet \to \Omega^\bullet$, as given in \cref{def:almost}.
We will do this in the next subsection.

\subsection{Finishing the proof}

Let us first briefly explain our strategy.
Since $\Omega^\bullet$ has dimension $M := \dim_\mathbb{C} (\lieg / \liep_S)$ (see below), we need to show that the map $L^{M - k} : \Omega^k \to \Omega^{2 M - k}$ is an isomorphism for all $k = 0, \cdots, M - 1$.
By applying the functor $\Phi$ from Takeuchi's categorical equivalence, it suffices to show the same property for $\Phi(L^{M - k}) : \Phi(\Omega^k) \to \Phi(\Omega^{2 M - k})$.
We will show that this can be checked upon choosing an appropriate basis of $\algquo$.

Recall that for an irreducible flag manifold we have $S = \Pi \backslash \{\alpha_s\}$ for some $s \in \{1, \cdots, r\}$.
We will write $I := \{ 1, \cdots, N \}$, where $N = \dim V(\omega_s)$, and define the index set
\[
\IndP := \{ i \in I : (\omega_s, \omega_s - \alpha_s - \wt v_i) = 0 \}.
\]
As in \cite[Section 3.2.1]{heko06}, the elements of $\IndP$ label certain $\UqlS$-submodules of $V(\omega_s)$ and its dual.
It is known that $\# \IndP = \dim_\mathbb{C} (\lieg / \liep_S)$ and we will write $M = \dim_\mathbb{C} (\lieg / \liep_S)$.

Now consider $\Phi(\Omega^\bullet) = \takalg$, as in Takeuchi's categorical equivalence.
It turns out that $\Phi(\Omega^\bullet)$ inherits an algebra structure from $\Omega^\bullet$. This is because $\calB^+ \Omega^\bullet = \Omega^\bullet \calB^+$, as explained in \cite[Section 3.3.4]{{heko06}}.
Let us introduce the notation
\[
x_i := [\del z^{i N}], \quad
y_i := [\delbar z^{N i}], \quad
i \in \IndP,
\]
where we use $[\cdot]$ to denote the equivalence classes in the quotient $\takalg$.

\begin{lemma}
\label{lem:image-Phi}
We have
\[
[\del z^{i j}] = \begin{cases}
x_i & i \in \IndP, \ j = N \\
0 & \mathrm{otherwise}
\end{cases}
, \quad
[\delbar z^{i j}] = \begin{cases}
y_i & i = N, \ j \in \IndP \\
0 & \mathrm{otherwise}
\end{cases}
.
\]
\end{lemma}

\begin{proof}
This is shown in the proofs of \cite[Proposition 3.3]{{heko06}} and \cite[Proposition 3.4]{{heko06}}.
It can also be easily deduced from \cite[Lemma 8]{{heko04}}.
\end{proof}

The algebra $\Phi(\Omega^\bullet)$ can be endowed with a filtration which can be used to show that $\dim \Phi(\Omega^k) = \binom{2 M}{k}$, as for the classical exterior algebra. Moreover, a vector space basis for $\Phi(\Omega^\bullet)$ can be obtained by taking appropriate products of the generators $\{ x_i, y_i : i \in \IndP \}$.
For details on this filtration see \cite[Sections 3.3.1 and 3.3.4]{{heko06}}.

Next we will show that we can choose a basis with a particularly nice property, as explained in the next lemma.
We denote by $m$ be the smallest positive integer such that $m (P, P) \subset \mathbb{Z}$, where $P$ is the weight lattice.
Moreover recall the following terminology: given an algebra $A$ and a vector space basis $\{e_i\}_i$, the \emph{structure constants} are the coefficients $c_{i j}^k$ appearing in the expansion $e_i \cdot e_j = \sum_k c_{i j}^k e_k$ with respect to the given basis.

\begin{lemma}
\label{lem:basis-lau}
We can choose a basis for the algebra $\Phi(\Omega^\bullet)$ in such a way that the structure constants are in $\mathbb{Z}[q^{1/m}, q^{-1/m}]$ for all $0 < q \leq 1$.
\end{lemma}

\begin{proof}
Denote by $\UqgZ$ the integral form of $\Uqg$, see for instance \cite[Chapter 9.2]{chpr95}.
Fix a basis $\{v_i\}_{i \in I}$ of $V(\omega_s)$ in such a way that the $\mathbb{Z}[q, q^{-1}]$-module generated by the basis $\{v_i\}_{i \in I}$ is invariant under $\UqgZ$.
Then the same holds for the dual basis $\{f_i\}_{i \in I}$ of $V(\omega_s)^*$.
Moreover recall that the linear span of the generators $\{ x_i : i \in \IndP \}$ and $\{ y_i : i \in \IndP \}$ can be identified with appropriate $\UqlS$-submodules of $V(\omega_s)^*$ and $V(\omega_s)$, respectively.

Consider the subalgebra of $\algquo$ generated by the $\{x_i : i \in \IndP \}$.
The relations for this algebra are given in \cite[Proposition 3.6 (ii)]{heko06}.
Upon carefully analyzing this proof, we find that we can choose a basis $\{x_A\}_A$ in such a way the structure constants are in $\mathbb{Z}[q, q^{-1}]$.
A similar result holds for the subalgebra of $\algquo$ generated by the $\{y_i : i \in \IndP \}$.

Next we look at the cross-relations between the $\{x_i : i \in \IndP \}$ and the $\{y_i : i \in \IndP \}$. These appear in \cite[Proposition 3.11 (ii)]{heko06} and given in terms of the braiding $\hat{R}_{V, V^*}$, up to a prefactor.
With our choice of bases it can be shown that the matrix coefficients of $\hat{R}_{V, V^*}$ are in $\mathbb{Z}[q^{1/m}, q^{-1/m}]$, see for instance \cite[Lemma 4.10]{chtu14} and the discussion in that section.
Fractional powers of $q$ are required since the $R$-matrix takes the form $R = \tilde{R} \cdot D$, where $D(v_\lambda \otimes v_\mu) = q^{(\lambda, \mu)} v_\lambda \otimes v_\mu$ for weight vectors while $\tilde{R}$ has matrix coefficients in $\mathbb{Z}[q, q^{-1}]$.

Putting all together, we have shown that we can choose a basis of $\Phi(\Omega^\bullet)$ in such a way that the structure constants are in $\mathbb{Z}[q^{1/m}, q^{-1/m}]$ for all $0 < q < 1$. Since Laurent polynomials in $q^{1/m}$ are continuous for $q > 0$, the result also extends to $q = 1$.
\end{proof}

We are now ready to show that $\kappa$ is an almost symplectic form, as in \cref{def:almost}.

\begin{proposition}
\label{prop:lefschetz-iso}
The map $L^{M - k} : \Omega^k \to \Omega^{2 M - k}$ is an isomorphism for all $k = 0, \cdots, M - 1$, except possibly for finitely many values of $q$.
In particular this is true when $q$ is transcendental.
\end{proposition}

\begin{proof}
Using Takeuchi's categorical equivalence, it is enough to prove the corresponding result for the linear map $\Phi(L^{M - k}) : \Phi(\Omega^k) \to \Phi(\Omega^{2 M - k})$.
Moreover, as $\Phi(\Omega^k)$ and $\Phi(\Omega^{2 M - k})$ have the same dimension, it suffices to show that $\Phi(L^{M - k})$ is injective.

We fix a basis for $\algquo$ as in \cref{lem:basis-lau} and we denote by $V^\bullet$ the underlying vector space.
In this way we have the same vector space $V^\bullet$ for all $0 < q \leq 1$, and we consider the multiplication of $\algquo$ as a linear map $V^\bullet \otimes V^\bullet \to V^\bullet$ depending on the parameter $q$.

In the same manner the map $\Phi(L^{M - k})$ can be seen as a linear map $V^k \to V^{2 M - k}$ depending on the parameter $q$, since it is given by the formula $\Phi(L^{M - k}) v = \Phi(\kappa^{M - k}) \wedge v$ for $v \in \Phi(\Omega^k)$.
In particular, using the formulae given in \cref{lem:image-Phi}, we obtain
\[
\Phi(\kappa) = \iu \sum_{i \in \IndP} q^{(2 \rho, \lambda_i)} y_i \wedge x_i.
\]
Let us write $\{ v_i^{(k)} \}_i$ for the fixed basis of $V^k$ for $k = 0, \cdots, 2 M$. Then we have
\[
\Phi(L^{M - k}) v_i^{(k)} = \sum_j c_i^j(q) v_j^{(2M - k)}.
\]
By \cref{lem:basis-lau} we have that all the coefficients $c_i^j(q)$ are in $\laupol$, since the structure constants for the multiplication in $\algquo$ satisfy this property.

Now let us write $M(q)$ for the matrix representing the linear map $\Phi(L^{M - k}): V^k \to V^{2 M - k}$.
Since $M(q)$ has entries $\{ c_i^j(q) \}_{i, j}$, it follows that $\det M(q)$ is a Laurent polynomial in $q^{1 / m}$.
Suppose that $\det M(1) \neq 0$, so that $\det M(q)$ is not the zero polynomial.
Under this assumption, $\det M(q)$ can only vanish for finitely many values of $q$.
Moreover, as the polynomial $\det M(q)$ has coefficients in $\mathbb{Z}$, it is never zero when $q$ is transcendental.

Hence we only need to show that $\det M(1) \neq 0$, that is the map $\Phi(L^{M - k}): V^k \to V^{2 M - k}$ is an isomorphism for $q = 1$.
In the classical limit the algebra $\algquo$ becomes isomorphic to the exterior algebra $\bigwedge(W_\mathbb{C})$, where $W = \mathrm{span}_\mathbb{R} \{ x_i, y_i : i \in \IndP \}$ is a real vector space of dimension $2 M$.
The vector space $W$ carries a natural symplectic structure and $\Phi(\kappa) = \iu \sum_{i \in \IndP} y_i \wedge x_i$ is its canonical form (up to a scalar).
Finally it follows from classical results that the map $\Phi(L^{M - k}) : V^k \to V^{2 M - k}$ is an isomorphism, see for instance \cite[Chapter 1.2]{huy}.
\end{proof}

Using this fact, it is immediate to prove the main result of this section.

\begin{theorem}
\label{thm:kahler}
The pair $(\Omega^{(\bullet, \bullet)}, \kappa)$ is a Kähler structure for $\Omega^\bullet$, except possibly for finitely many values of $q$.
In particular this is true when $q$ is transcendental.
\end{theorem}

\begin{proof}
This follows by combining \cref{prop:kappa-props} and \cref{prop:lefschetz-iso}.
We have that $\kappa$ is a central, real 2-form such that $L^{M -k}: \Omega^k \to \Omega^{2 M - k}$ is an isomorphism for $k = 0, \cdots, M - 1$, hence is an almost symplectic form according to \cref{def:almost}.
Next $\Omega^\bullet$ has a natural complex structure and $\kappa \in \Omega^{(1, 1)}$, hence it is a Hermitian form according to \cref{def:hermitian}.
Finally $\kappa$ is $\diff$-closed and hence a Kähler form according to \cref{def:kahler}.
\end{proof}

This proves the conjecture formulated in \cite[Conjecture 4.25]{obu17}, with the possible exception of finitely many values of $q$.
Since for the quantum projective spaces this result is valid for all $0 < q < 1$, we expect this to be the case in full generality.

\appendix

\section{Some identities for the braiding}
\label{sec:braiding}

In this appendix we will derive some identities for the components of the braiding, which are related to its behaviour under the operations of duality and adjoint.
These are used in the main text to prove the compatibility of the differential calculi with the $*$-structure.

\subsection{Basic facts}

Let $V$ and $W$ be finite-dimensional $\Uqg$-modules. Fix bases $\{v_i\}_i$ of $V$ and $\{w_i\}_i$ of $W$.
Then for the braiding $\hat{R}_{V, W} : V \otimes W \to W \otimes V$ we write
\[
\hat{R}_{V, W} (v_i \otimes w_j) = \sum_{k, l} (\hat{R}_{V, W})^{k l}_{i j} w_k \otimes v_l.
\]

Consider the double dual $\Uqg$-module $V^{**}$. To any vector $v \in V$ we can associate the linear functional $\widetilde{v} \in V^{**}$ on $V^*$ defined by $\widetilde{v}(f) := f(v)$.
The map $v \mapsto \widetilde{v}$ is an isomorphism of vector spaces but not of $\Uqg$-modules.
On the other hand, it is simple to check that the map $\eta_V : V \to V^{**}$ given by $\eta_V(v) := \widetilde{K_{2 \rho} v}$ is an isomorphism of $\Uqg$-modules.
This is because in our conventions we have $S^2(X) = K_{2 \rho} X K_{2 \rho}^{-1}$ for all $X \in \Uqg$.

We will also need $\Uqg$-invariant inner products.
Fixing $(\cdot, \cdot)_V$ on $V$, we will write $j_V : V \to V^*$ for the conjugate-linear map given by $j_V(v)(w) := (v, w)_V$. Then
\[
(j_V(v), j_V(w))_{V^*} := (K_{2 \rho} w, v)_V
\]
is an invariant inner product on $V^*$.
To check this claim one uses the fact that $(\cdot, \cdot)_V$ is invariant and the identities $S^2(X) = K_{2 \rho} X K_{2 \rho}^{-1}$ and $S(X)^* = S^{-1}(X^*)$.

\subsection{The identities}

We will now derive some identities for the components of the braiding $\hat{R}_{V, W}$ under the operations of duality and adjoint.

\begin{lemma}
\label{lem:dual-rels}
We have the identities
\[
(\hat{R}_{V^*, W})_{i j}^{k l} = (\hat{R}_{V, W}^{-1})_{j l}^{i k}, \quad
(\hat{R}_{W, V^*}^{-1})_{i j}^{k l} = (\hat{R}_{W, V})_{j l}^{i k}.
\]
\end{lemma}

\begin{proof}
We recall a general result valid for (strict) braided monoidal categories with duals from \cite{egno16}.
Let $\mathcal{C}$ be a braided monoidal category with braiding $c$.
Let $X, Y$ be objects of $\mathcal{C}$ and let $X^*$ be the left dual of $X$.
Then according to \cite[Lemma 8.9.1]{egno16} we have
\begin{equation}
\label{eq:egno-rels}
\begin{split}
c_{X^*, Y} & = (\mathrm{ev}_X \otimes \id_{Y \otimes X^*}) \circ (\id_{X^*} \otimes c_{X, Y}^{-1} \otimes \id_{X^*}) \circ (\id_{X^* \otimes Y} \otimes \mathrm{coev}_X), \\
c_{Y, X^*}^{-1} & = (\mathrm{ev}_X \otimes \id_{Y \otimes X^*}) \circ (\id_{X^*} \otimes c_{Y, X} \otimes \id_{X^*}) \circ (\id_{X^* \otimes Y} \otimes \mathrm{coev}_X).
\end{split}
\end{equation}
Here $\mathrm{ev_X}: X^* \otimes X \to 1$ and $\mathrm{coev_X}: 1 \to X \otimes X^*$ are the evaluation and coevaluation morphisms associated with the left dual object $X^*$.

We apply this to the category of finite-dimensional $\Uqg$-modules.
In this case, given a $\Uqg$-module $V$, the evaluation and coevaluation morphisms are given by
\[
\mathrm{ev}_V (f \otimes v) = f(v), \quad
\mathrm{coev}_V (1) = \sum_i v_i \otimes f_i.
\]
Here $v \in V$ and $f \in V^*$, while $\{v_i\}_i$ is a basis of $V$ and $\{f_i\}_i$ is a dual basis of $V$.

Now consider $\hat{R}_{V^*, W}: V^* \otimes W \to W \otimes V$.
Using the first relation of \eqref{eq:egno-rels} we compute
\[
\begin{split}
\hat{R}_{V^*, W} (f_i \otimes w_j) & = (\mathrm{ev}_V \otimes \id_{W \otimes V^*}) \circ (\id_{V^*} \otimes \hat{R}_{V, W}^{-1} \otimes \id_{V^*}) (f_i \otimes w_j \otimes \sum_k v_k \otimes f_k) \\
& = \sum_k (\mathrm{ev}_V \otimes \id_{W \otimes V^*}) (f_i \otimes \sum_{a, b} (\hat{R}_{V, W}^{-1})_{j k}^{a b} v_a \otimes w_b \otimes f_k) \\
& = \sum_{b, k} (\hat{R}_{V, W}^{-1})_{j k}^{i b} w_b \otimes f_k.
\end{split}
\]
Comparing this to $\hat{R}_{V^*, W} (f_i \otimes w_j) = \sum_{b, k} (\hat{R}_{V^*, W})_{i j}^{b k} w_b \otimes f_k$ we get the first identity.
The second identity follows from the second relation of \eqref{eq:egno-rels} in a similar way.
\end{proof}

Next we derive some identities involving double duals.

\begin{lemma}
\label{lem:ddual-rels}
We have the identities
\[
(\hat{R}_{V^{**}, W})_{i j}^{k l} = q^{(2 \rho, \wt(v_l) - \wt(v_i))} (\hat{R}_{V, W})_{i j}^{k l}, \quad
(\hat{R}_{V, W^{**}})_{i j}^{k l} = q^{(2 \rho, \wt(w_k) - \wt(w_j))} (\hat{R}_{V, W})_{i j}^{k l}.
\]
\end{lemma}

\begin{proof}
Consider the $\Uqg$-module isomorphism $\eta_V: V \to V^{**}$ defined previously.
Then using naturality of the braiding $\hat{R}_{V, W}$ we obtain
\[
\hat{R}_{V^{**}, W} = (\id_W \otimes \eta_V) \circ \hat{R}_{V, W} \circ (\eta_V^{-1} \otimes \id_W).
\]
Using the fact that $\eta_V(v_i) = q^{(2 \rho, \wt(v_i))} \widetilde{v_i}$ we compute
\[
\begin{split}
\hat{R}_{V^{**}, W} (\widetilde{v_i} \otimes w_j) & = (\id_W \otimes \eta_V) \circ \hat{R}_{V, W} (q^{-(2 \rho, \wt(v_i))} v_i \otimes w_j) \\
& = q^{-(2 \rho, \wt(v_i))} (\id_W \otimes \eta_V) (\sum_{k, l} (\hat{R}_{V, W})_{i j}^{k l} w_k \otimes v_l) \\
& = q^{-(2 \rho, \wt(v_i))} q^{(2 \rho, \wt(v_i))} \sum_{k, l} (\hat{R}_{V, W})_{i j}^{k l} w_k \otimes \widetilde{v_l}.
\end{split}
\]
This gives the first identity.
The second identity is proven similarly by considering the identity $\hat{R}_{V, W^{**}} = (\eta_W \otimes \id_V) \circ \hat{R}_{V, W} \circ (\id_V \otimes \eta_W^{-1})$, which also follows by naturality.
\end{proof}

Finally we derive some identities which combine those derive above with complex conjugation.
These identities are used to prove the main result of \cref{sec:heko}.

\begin{proposition}
\label{prop:braiding-components}
Let $\{v_i\}_i$ be an orthonormal basis of $V$ and $\{f_i\}_i$ the dual basis of $V^*$.
\begin{enumerate}
\item We have $\overline{(\hat{R}_{V, V})^{k l}_{i j}} = (\hat{R}_{V^*, V^*})^{l k}_{j i}$.
\item We have $\overline{(\hat{R}_{V, V^*})^{k l}_{i j}} = (\hat{R}_{V, V^*})^{l k}_{j i}$.
\end{enumerate}
\end{proposition}

\begin{proof}
(1) We have the identity $\hat{R}_{V, W}^* = \hat{R}_{W, V}$, see for instance \cite[Example 2.6.4]{netu13}.
From this is easily follows that $\overline{(\hat{R}_{V, W})^{k l}_{i j}} = (\hat{R}_{W, V})^{i j}_{k l}$, provided that we use orthonormal bases for $V$ and $W$.
Using this fact and \cref{lem:dual-rels} we compute
\[
\overline{(\hat{R}_{V, V})^{k l}_{i j}} = (\hat{R}_{V, V})^{i j}_{k l} = (\hat{R}_{V, V^*}^{-1})^{j l}_{i k} = (\hat{R}_{V^*, V^*})^{l k}_{j i}.
\]

(2) First we observe that the dual basis $\{f_i\}_i$ is not orthonormal. Indeed, using the fact that $f_i = j_V(v_i)$ and the definition of the inner product on $V^*$, we compute
\[
(f_i, f_j)_{V^*} = (K_{2 \rho} v_j, v_i)_V = \delta_{i j} q^{(2 \rho, \wt v_i)}.
\]
Hence we obtain an orthonormal basis by setting $f_i^\prime = q^{-(\rho, \wt v_i)} f_i$.
We write the formulae for the braidings with respect to the orthonormal bases $\{v_i\}_i$ and $\{f_i^\prime\}_i$ as follows
\[
\hat{R}_{V, V^*} (v_i \otimes f_j^\prime) = \sum_{k, l} a^{k l}_{i j} f_k^\prime \otimes v_l, \quad
\hat{R}_{V^*, V} (f_i^\prime \otimes v_j) = \sum_{k, l} b^{k l}_{i j} v_k \otimes f_l^\prime.
\]
From these we immediately get the relations
\[
(\hat{R}_{V, V^*})^{k l}_{i j} = q^{(\rho, \wt v_j - \wt v_k)} a^{k l}_{i j}, \quad
(\hat{R}_{V^*, V})^{k l}_{i j} = q^{(\rho, \wt v_i - \wt v_l)} b^{k l}_{i j}.
\]
Then using the identity $\overline{(\hat{R}_{V, W})^{k l}_{i j}} = (\hat{R}_{W, V})^{i j}_{k l}$ for orthonormal bases we get
\[
\overline{(\hat{R}_{V, V^*})^{k l}_{i j}} = q^{(\rho, \wt v_j - \wt v_k)} \overline{a^{k l}_{i j}} = q^{(\rho, \wt v_j - \wt v_k)} b^{i j}_{k l} = q^{(2 \rho, \wt v_j - \wt v_k)} (\hat{R}_{V^*, V})^{i j}_{k l}.
\]
Using again the identities from \cref{lem:dual-rels} we get $(\hat{R}_{V^*, V})^{i j}_{k l} = (\hat{R}_{V^{**}, V^*})^{l k}_{j i}$.
Finally we get rid of the double dual by using \cref{lem:ddual-rels} and obtain
\[
\overline{(\hat{R}_{V, V^*})^{k l}_{i j}} = q^{(2 \rho, \wt v_j - \wt v_k)} (\hat{R}_{V^{**}, V^*})^{l k}_{j i} = (\hat{R}_{V, V^*})^{l k}_{j i}. \qedhere
\]
\end{proof}

\section{Differential calculus identities}
\label{sec:braidingcalc}

In this appendix we will derive some identities related to the differential calculus $\calc$.
According to \cite[Proposition 3.3 (ii)]{heko06} and \cite[Proposition 3.4 (ii)]{heko06}, the right $\subalg$-module structures of the FODCs $\calcDel$ and $\calcDelb$ are given by the formulae
\[
\del(z) z = q^{(\alpha_s, \alpha_s)} \braidz z \del(z), \quad
\delbar(z) z = q^{-(\alpha_s, \alpha_s)} \braidz z \delbar(z),
\]
where the indices are suppressed and we define the linear map
\[
\braidz := (\hat{R}_{V, V^*})_{2 3} (\hat{R}_{V, V})_{1 2} (\hat{R}_{V^*, V^*}^{-1})_{3 4} (\hat{R}_{V, V^*}^{-1})_{2 3}.
\]
Writing out the indices explicitly for the first relation, we have
\[
\del(z^{i j}) z^{k l} = q^{(\alpha_s, \alpha_s)} \sum_{a, b, c, d} T^{i j k l}_{a b c d} z^{a b} \del(z^{c d}).
\]
where the components of $\braidz$ are given by
\[
T_{a b c d}^{i j k l} = \sum_{p, q, r, s} (\hat{R}_{V, V^*})_{p q}^{j k} (\hat{R}_{V, V})_{a r}^{i p} (\hat{R}_{V^*, V^*}^{-1})_{s d}^{q l} (\hat{R}_{V, V^*}^{-1})_{b c}^{r s}.
\]
Our aim is to derive some identities for the map $\braidz$.

\begin{lemma}
\label{lem:idT1}
We have the identity
\[
\sum_{j, k, l} \braidz_{a b c d}^{i j k l} \braidz_{k l c^\prime d^\prime}^{j j^\prime k^\prime l^\prime} = \delta_{d c^\prime} \braidz_{a b c d^\prime}^{i j^\prime k^\prime l^\prime}.
\]
\end{lemma}

\begin{proof}
Writing out the left-hand side of the equation we have
\[
\begin{split}
\sum_{j, k, l} \braidz_{a b c d}^{i j k l} \braidz_{k l c^\prime d^\prime}^{j j^\prime k^\prime l^\prime} & = \sum_{j, k, l} \sum_{p, q, r, s} (\hat{R}_{V, V^*})_{p q}^{j k} (\hat{R}_{V, V})_{a r}^{i p} (\hat{R}_{V^*, V^*}^{-1})_{s d}^{q l} (\hat{R}_{V, V^*}^{-1})_{b c}^{r s} \\
& \times \sum_{p^\prime, q^\prime, r^\prime, s^\prime} (\hat{R}_{V, V^*})_{p^\prime q^\prime}^{j^\prime k^\prime} (\hat{R}_{V, V})_{k r^\prime}^{j p^\prime} (\hat{R}_{V^*, V^*}^{-1})_{s^\prime d^\prime}^{q^\prime l^\prime} (\hat{R}_{V, V^*}^{-1})_{l c^\prime}^{r^\prime s^\prime}.
\end{split}
\]
Using \cref{lem:dual-rels} we get $(\hat{R}_{V, V})_{k r^\prime}^{j p^\prime} = (\hat{R}_{V, V^*}^{-1})_{j k}^{p^\prime r^\prime}$. Then summing over $j, k$ we obtain
\[
\begin{split}
\sum_{j, k, l} \braidz_{a b c d}^{i j k l} \braidz_{k l c^\prime d^\prime}^{j j^\prime k^\prime l^\prime} & = \sum_l \sum_{p, q, r, s} (\hat{R}_{V, V})_{a r}^{i p} (\hat{R}_{V^*, V^*}^{-1})_{s d}^{q l} (\hat{R}_{V, V^*}^{-1})_{b c}^{r s} \\
& \times \sum_{q^\prime, s^\prime} (\hat{R}_{V, V^*})_{p q^\prime}^{j^\prime k^\prime} (\hat{R}_{V^*, V^*}^{-1})_{s^\prime d^\prime}^{q^\prime l^\prime} (\hat{R}_{V, V^*}^{-1})_{l c^\prime}^{q s^\prime}.
\end{split}
\]
Next, we get $(\hat{R}_{V, V^*}^{-1})_{l c^\prime}^{q s^\prime} = (\hat{R}_{V^*, V^*})_{q l}^{s^\prime c^\prime}$ again by \cref{lem:dual-rels}. Summing over $q, l$ we obtain
\[
\sum_{j, k, l} \braidz_{a b c d}^{i j k l} \braidz_{k l c^\prime d^\prime}^{j j^\prime k^\prime l^\prime} = \delta_{d, c^\prime} \sum_{p, r, s} (\hat{R}_{V, V})_{a r}^{i p} (\hat{R}_{V, V^*}^{-1})_{b c}^{r s} \sum_{q^\prime} (\hat{R}_{V, V^*})_{p q^\prime}^{j^\prime k^\prime} (\hat{R}_{V^*, V^*}^{-1})_{s d^\prime}^{q^\prime l^\prime}.
\]
Upon relabeling we obtain the identity we were after.
\end{proof}

Finally we derive an identity for a particular contraction of the tensor $T$.

\begin{lemma}
\label{lem:idT2}
We have the identity
\[
\sum_{i, j} q^{(2 \rho, \lambda_i)} T_{a b c d}^{i j j i} = \delta_{a d} \delta_{b c} q^{(2 \rho, \lambda_a)}.
\]
\end{lemma}

\begin{proof}
First of all we have
\[
\sum_{i, j, k, m} q^{(2 \rho, \lambda_i + \lambda_m)} (\hat{R}_{V, V^*}^{-1})_{j k}^{m m} T_{a b c d}^{i j k i} = \sum_{i, m} \sum_{r, s} q^{(2 \rho, \lambda_i + \lambda_m)} (\hat{R}_{V, V})_{a r}^{i m} (\hat{R}_{V^*, V^*}^{-1})_{s d}^{m i} (\hat{R}_{V, V^*}^{-1})_{b c}^{r s}.
\]
Next using \cref{lem:dual-rels} we can rewrite
\[
(\hat{R}_{V^*, V^*}^{-1})_{s d}^{m i} = (\hat{R}_{V^*, V})_{d i}^{s m} = (\hat{R}_{V, V}^{-1})_{i m}^{d s}.
\]
Also by weight reasons we have $(\hat{R}_{V, V})_{a r}^{i m} = 0$ unless $\lambda_i + \lambda_m = \lambda_a + \lambda_r$.
Then we obtain
\[
\sum_{i, j, k, m} q^{(2 \rho, \lambda_i + \lambda_m)} (\hat{R}_{V, V^*}^{-1})_{j k}^{m m} T_{a b c d}^{i j k i} = \delta_{a, d} q^{(2 \rho, \lambda_a)} \sum_r q^{(2 \rho, \lambda_r)} (\hat{R}_{V, V^*}^{-1})_{b c}^{r r}.
\]
Now suppose that $\sum_m q^{(2 \rho, \lambda_m)} (\hat{R}_{V, V^*}^{-1})_{j k}^{m m} = c \delta_{j k}$ for some $c \neq 0$.
Plugging this identity in the previous equation gives the claim of the lemma, hence it suffices to prove that it holds.

Let us consider the following vectors
\[
I_{V, V^*} = \sum_i v_i \otimes f_i \in V \otimes V^*, \quad
J_{V^*, V} = \sum_i q^{- (2 \rho, \lambda_i)} f_i \otimes v_i \in V^* \otimes V.
\]
It is easy to see that they are $\Uqg$-invariant.
Since $V$ is irreducible, we must have
\[
\hat{R}_{V^*, V}^{-1} (I_{V, V^*}) = c J_{V^*, V}, \quad
c \in \mathbb{C}.
\]
Moreover $c \neq 0$ since $\hat{R}_{V^*, V}$ is invertible. In components this gives the identity
\begin{equation}
\label{eq:id-tbraid}
\sum_i (\hat{R}_{V^*, V}^{-1})_{i i}^{k l} = \delta^{k l} c q^{- (2 \rho, \lambda_k)}.
\end{equation}
On the other hand using \cref{lem:dual-rels} and \cref{lem:ddual-rels} we can rewrite
\[
(\hat{R}_{V, V^*}^{-1})_{i j}^{k l} = (\hat{R}_{V^*, V^*})_{k i}^{l j} = (\hat{R}_{V^*, V^{**}}^{-1})_{l k}^{j i} = q^{(2 \rho, \lambda_i - \lambda_l)} (\hat{R}_{V^*, V}^{-1})_{l k}^{j i}.
\]
Finally using this and \eqref{eq:id-tbraid} we obtain
\[
\sum_m q^{(2 \rho, \lambda_m)} (\hat{R}_{V, V^*}^{-1})_{i j}^{m m} = q^{(2 \rho, \lambda_i)} \sum_m (\hat{R}_{V^*, V}^{-1})_{m m}^{j i} = \delta_{i j} c q^{(2 \rho, \lambda_i)}.
\]
This is the identity we wanted to establish, which concludes the proof.
\end{proof}

\end{document}